\documentclass[10pt]{amsart}
\usepackage{amssymb}
\usepackage{amsthm,amsmath}
\usepackage{cite}

\usepackage{MnSymbol}

\title[Irreducible Polyadic Semigroups]{Irreducible Polyadic Semigroups Admitting the Adjunction of a Neutral Element}

\author{Jean-Luc Marichal}
\address{University of Luxembourg, Department of Mathematics, Maison du Nombre, 6, avenue de la Fonte, L-4364 Esch-sur-Alzette, Luxembourg}
\email{jean-luc.marichal@uni.lu}

\author{Pierre Mathonet}
\address{University of Li\`ege, Department of Mathematics, All\'ee de la D\'ecouverte, 12 - B37, B-4000 Li\`ege, Belgium}
\email{p.mathonet@uliege.be}

\author{Tam\'as Waldhauser}%\thanks{Corresponding author: Tam\'as Waldhauser, University of Szeged, Bolyai Institute, Aradi v\'ertan\'uk tere 1, H-6720 Szeged, Hungary. Email: twaldha@math.u-szeged.hu}
\address{University of Szeged, Bolyai Institute, Aradi v\'ertan\'uk tere 1, H-6720 Szeged, Hungary}
\email{twaldha@math.u-szeged.hu}

\date{September 16, 2025}

\begin{document}

\theoremstyle{plain}

\newtheorem{theorem}{Theorem}[section]
\newtheorem{lemma}[theorem]{Lemma}
\newtheorem{proposition}[theorem]{Proposition}
\newtheorem{corollary}[theorem]{Corollary}
\newtheorem{fact}[theorem]{Fact}

\theoremstyle{definition}

\newtheorem{definition}[theorem]{Definition}

\newtheorem{ex}[theorem]{Example} % \newtheorem establishes the object heading
    \newenvironment{example}    % this is the environment name for the input
    {\renewcommand{\qedsymbol}{$\lozenge$}%
    \pushQED{\qed}\begin{ex}}
    {\popQED\end{ex}}

\newtheorem{rem}[theorem]{Remark} % \newtheorem establishes the object heading
    \newenvironment{remark}    % this is the environment name for the input
    {\renewcommand{\qedsymbol}{$\lozenge$}%
    \pushQED{\qed}\begin{rem}}
    {\popQED\end{rem}}

\theoremstyle{remark}

\newtheorem{claim}{Claim}

\newcommand{\R}{\mathbb{R}}
\newcommand{\N}{\mathbb{N}}
\newcommand{\Z}{\mathbb{Z}}
\newcommand{\id}{\mathrm{id}}

\newcommand{\wt}[1]{\marginpar{\hrule\medskip Tamás: #1}}
\newcommand{\abs}[1]{\left\vert #1 \right\vert}

\begin{abstract}
It was claimed in \cite{DudMuk06} that for any integer $n\geqslant 2$, a neutral element can be adjoined to an $n$-ary semigroup if and only if the $n$-ary semigroup is reducible to a binary semigroup. We show that the `only if' direction of this statement is incorrect when $n$ is odd. Moreover, we offer a comprehensive characterization of the class of irreducible $n$-ary semigroups, for odd $n$, that admit the adjunction of a neutral element.
\end{abstract}



\maketitle

%------------------------------------------------------------------------------------
\section{Introduction}

Let $X$ be a nonempty set and let $n\geqslant 2$ be an integer. Recall that an $n$-ary operation $F\colon X^n \to X$ is said to be \emph{associative} if it satisfies the following system of functional equations:
\begin{eqnarray*}
\lefteqn{F(x_1,\ldots,x_{i-1},F(x_i,\ldots,x_{i+n-1}),x_{i+n},\ldots,x_{2n-1})}\\
&=& F(x_1,\ldots,x_i,F(x_{i+1},\ldots,x_{i+n}),x_{i+n+1},\ldots,x_{2n-1})
\end{eqnarray*}
for all $x_1,\ldots,x_{2n-1}\in X$ and all $1\leqslant i\leqslant n-1$. The pair $(X,F)$ is commonly known as an \emph{$n$-ary semigroup}, or alternatively, a \emph{polyadic semigroup} when the arity $n$ is not specified. This concept was introduced by Dörnte \cite{Dor28} and has subsequently led to the development of the notion of an $n$-ary group, which was first studied by Post \cite{Pos40}.

An associative $n$-ary operation $F\colon X^n\to X$ is said to be \emph{reducible to} (or \emph{derived from}) an associative binary operation $\circ\,\colon X^2\to X$ if it satisfies the following equation (where $\circ$ is used in its standard infix notation):
\begin{eqnarray*}
F(x_1,x_2,x_3\ldots,x_n) &=& (\, \cdots\, ((x_1\circ x_2)\circ x_3) \circ \,\cdots \, )\circ x_n\\
&=& x_1\circ x_2\circ x_3\circ\,\cdots\,\circ x_n \qquad (x_1,x_2,x_3,\ldots,x_n\in X).
\end{eqnarray*}
In this case, the $n$-ary semigroup $(X,F)$ is also said to be \emph{reducible to} (or \emph{derived from}) the semigroup $(X,\circ)$.

Let us provide a straightforward example of an irreducible ternary semigroup.

\begin{example}\label{ex:ternary1}
The real associative operation $F\colon\R^3\to\R$, defined by
$$
F(x_1,x_2,x_3) ~=~ x_1-x_2+x_3\qquad\text{for $x_1,x_2,x_3\in\R$},
$$
is irreducible. Indeed, suppose that this operation is reducible to an associative binary operation $\circ\,\colon \R^2\to \R$ and let $c=0\circ 0$. Then, for any $x\in\R\setminus\{c\}$, we have $c\circ x=x=x\circ c$ and hence also $x=c\circ x\circ c=2c-x$, which implies $x=c$, a contradiction.
\end{example}

Recall also that an element $e\in X$ is said to be \emph{neutral} for an $n$-ary operation $F\colon X^n\to X$ (or an $n$-ary semigroup $(X,F)$) if, for any $x\in X$, the following condition is satisfied:
$$
F(e^{k-1},x,e^{n-k}) ~=~ x\qquad (k=1,\ldots,n).
$$
Here and throughout, for any $k\in\{0,\ldots,n\}$ and any $x\in X$, the symbol $x^k$ represents the list $x,\ldots,x$, with $k$ copies of $x$. For instance, for any $x,y,z\in X$, we have
$$
F(x^3,y^0,z^2) ~=~ F(x,x,x,z,z).
$$
Recall that, when $n=2$, the neutral element is unique and often also called an {\em identity} element.

Dudek and Mukhin \cite[Proposition~2]{DudMuk06} made the following appealing and intriguing claim: ``\emph{one can adjoin a neutral element to an $n$-ary semigroup if and only if the $n$-ary semigroup is reducible to a binary semigroup}'' (see Definition~\ref{de:eAug22} below for the formal definition of adjoining a neutral element). 

The main objective of this paper is to present a counter-example that unfortunately disproves the `only if' direction of their statement when $n$ is an odd integer. This means that \emph{there exists at least one irreducible $n$-ary semigroup, where $n$ is odd, that admits the adjunction of a neutral element}.

The structure of the paper is as follows. In Section 2, we revisit the `if' direction of Dudek and Mukhin’s statement and establish that the `only if' direction holds when $n$ is even. In Section 3, we provide a comprehensive characterization of the class of irreducible $n$-ary semigroups, for odd $n$, that admit the adjunction of a neutral element. Section 4 presents a method for constructing such $n$-ary semigroups, along with explicit examples that serve to refute the `only if' direction when $n$ is odd.

%------------------------------------------------------------------------------------
\section{Preliminaries}

Let us review two important results given by Dudek and Mukhin. We begin by examining the following straightforward lemma, which complements \cite[Lemma~1]{DudMuk06}.

\begin{lemma}\label{lemma:21-firstL4}
If an $n$-ary semigroup $(X,F)$ has a neutral element $e\in X$, then the structure $(X,\circ)$ defined by the equation
$$
x\circ y ~=~ F(x,e^{n-2},y)\qquad\text{for $x,y\in X$},
$$
is a semigroup and, for any $k\in\{1,\ldots,n\}$, we have
$$%\begin{equation}\label{eq:121-firstL4}
x_1\circ\,\cdots\,\circ x_k ~=~ F(x_1,\ldots,x_k,e^{n-k})\qquad (x_1,\ldots,x_k\in X).
$$%\end{equation}
In particular, $(X,F)$ is reducible to the semigroup $(X,\circ)$; this semigroup also has $e$ as its neutral element and is the unique reduction of $(X,F)$ having this property.
\end{lemma}

\begin{proof}
We first observe that the operation $\circ$ defined in the statement is associative. In fact, for any $x,y,z\in X$, we have
\begin{eqnarray*}
(x\circ y)\circ z &=& F(F(x,e^{n-2},y),e^{n-2},z)\\
&=& F(x,e^{n-2},F(y,e^{n-2},z)) ~=~ x\circ (y\circ z).
\end{eqnarray*}
Moreover, it is clear that $e$ is also the neutral element for $\circ$.

Let us now prove the stated formula by induction on $k$. The identity holds tri{\-}vially for $k=1$ since $e$ is a neutral element for $F$. Assume now that it holds for some $1\leqslant k\leqslant n-1$, and let us show that it also holds for $k+1$. By using associativity and the fact that $e$ is a neutral element for $F$, we obtain
\begin{eqnarray*}
(x_1\circ\cdots\circ x_k)\circ x_{k+1} &=& F(x_1\circ\cdots\circ x_k,e^{n-2},x_{k+1})\\
&=& F(F(x_1\circ\cdots\circ x_k,e^{n-2},x_{k+1}),e^{n-1})\\
&=& F(x_1\circ\cdots\circ x_k,F(e^{n-2},x_{k+1},e),e^{n-2})\\
&=& F(x_1\circ\cdots\circ x_k,x_{k+1},e^{n-2})\\
&=& F(F(x_1,\ldots,x_k,e^{n-k}),x_{k+1},e^{n-2})\\
&=& F(x_1,\ldots,x_k,F(e^{n-k},x_{k+1},e^{k-1}),e^{n-k-1})\\
&=& F(x_1,\ldots,x_k,x_{k+1},e^{n-k-1}),
\end{eqnarray*}
which shows that the identity still holds for $k+1$. Taking $k=n$ immediately establishes the reducibility of $(X,F)$.

Now, suppose that $(X,\diamond)$ is a reduction of $(X,F)$ that has $e$ as the neutral element. Then, for any $x,y\in X$, we have
$$
x\diamond y ~=~ \underbrace{x\diamond e\diamond\,\cdots\,\diamond e\diamond y}_{n} ~=~ F(x,e^{n-2},y) ~=~ x\circ y.
$$
Therefore, the semigroup $(X,\circ)$ is the unique reduction of $(X,F)$ that has $e$ as the neutral element.
\end{proof}

It is well known and easy to verify (see, e.g., Clifford and Preston \cite[p.~4]{CliPre61}) that a neutral element can always be adjoined to a semigroup. In the next proposition, we restate \cite[Proposition~1]{DudMuk06} and its proof, using our notation; it asserts that this property also holds for any reducible $n$-ary semigroup. We begin with a formal definition.

\begin{definition}\label{de:eAug22}
We say that \emph{a neutral element can be adjoined to an $n$-ary semigroup} $(X,F)$ if there exists an $n$-ary semigroup $(X^*,F^*)$, where $X^*=X\cup\{e\}$ for some $e\notin X$, such that the operation $F^*\colon (X^*)^n\to X^*$ satisfies the following two properties:
\begin{itemize}
\item $F^*|_{X^n}=F$;
\item $e$ is neutral for $F^*$.
\end{itemize}
\end{definition}

\begin{proposition}\label{prop:2NeImRe44}
A neutral element can be adjoined to any reducible $n$-ary semigroup.
\end{proposition}

\begin{proof}
Let $(X,F)$ be an $n$-ary semigroup that is reducible to a semigroup $(X,\circ)$ and let $(X^*,\ast)$ be the semigroup obtained from $(X,\circ)$ by adjoining a neutral element $e\notin X$. One can readily see that the $n$-ary structure $(X^*,F^*)$, where the $n$-ary operation $F^*\colon (X^*)^n\to X^*$ is defined by the equation
$$
F^*(x_1,\ldots,x_n) ~=~ x_1\ast\,\cdots\,\ast x_n\qquad\text{for $x_1,\ldots,x_n\in X^*$},
$$
is a (reducible) $n$-ary semigroup with $e$ as a neutral element, whose restriction to $X$ is the $n$-ary semigroup $(X,F)$. 
\end{proof}

\begin{remark}
If an $n$-ary semigroup has a neutral element $e$, then it is reducible by Lemma~\ref{lemma:21-firstL4}. Moreover, by Proposition~\ref{prop:2NeImRe44}, we can always adjoin a neutral element to it. However, the element $e$ need not remain a neutral element for the resulting semigroup.
\end{remark}

The following example demonstrates that it is not possible to adjoin a neutral element to the real ternary semigroup mentioned in Example~\ref{ex:ternary1}.

\begin{example}\label{ex:ternary2}
The irreducible ternary semigroup $(\R,F)$, defined by
$$
F(x_1,x_2,x_3) ~=~ x_1-x_2+x_3\qquad\text{for $x_1,x_2,x_3\in\R$},
$$
does not admit the adjunction of a neutral element. Indeed, suppose on the contrary that we can adjoin a neutral element $e\notin\R$ to it. By Lemma~\ref{lemma:21-firstL4}, the resulting semigroup is then the ternary extension of a unique monoid $(X^*,\ast)$, with $X^*=\R\cup\{e\}$ and neutral element $e$. Now, let $u\in\R$ and define $v=u\ast u$. For any $x\in\R\setminus\{v\}$, we then have
$$
x\ast v ~=~ x\ast u\ast u ~=~ F(x,u,u) ~=~ x ~=~ F(u,u,x) ~=~ u\ast u\ast x ~=~ v\ast x,
$$
and hence
$$
x ~=~ v\ast x\ast v.
$$
If $v\in\R$, then this implies $x=F(v,x,v)=2v-x$, which is impossible.
This proves that $u \ast u = e$, i.e., the square of every element of $X^\ast$ is the neutral element $e$, thus $(X^*,\ast)$ is a group of exponent $2$.
Such groups are commutative, hence 
$$
y ~=~ F(x,x,y) ~=~ x \ast x \ast y ~=~ x \ast y \ast x ~=~ F(x,y,x) ~=~ 2x-y
$$ 
for all $x,y\in\R$, a contradiction.
\end{example}

In the following proposition, we show that, when $n$ is even, any $n$-ary semigroup that admits the adjunction of a neutral element is reducible. We first consider a definition and a lemma.

\begin{definition}
We say that an $n$-ary semigroup $(X,F)$ is an $n$-ary \emph{IN-semigroup} if it is irreducible and admits the adjunction of a neutral element.
\end{definition}

According to Examples~\ref{ex:ternary1} and \ref{ex:ternary2}, the real ternary semigroup $(\R,F)$ with $F(x_1,x_2,x_3) = x_1-x_2+x_3$ is irreducible, but not an IN-semigroup, thereby showing that not all irreducible $n$-ary semigroups are $n$-ary IN-semigroups.

\begin{lemma}\label{lemma:2ExtABe44}
Let $(X,F)$ be an $n$-ary IN-semigroup, let $(X^*,F^*)$ be the $n$-ary semigroup obtained from $(X,F)$ by adjoining a neutral element $e\notin X$, and let $(X^*,\ast)$ be the reduction of $(X^*,F^*)$ whose operation $\ast$ is defined by
$$
x\ast y ~=~ F^*(x,e^{n-2},y)\qquad (x,y\in X).
$$
Then, there exist $a,b\in X$ such that $a\ast b=e$.
\end{lemma}

\begin{proof}
We only need to prove that $X$ cannot be closed under the operation $\ast$. Suppose, on the contrary, that $X$ is closed under $\ast$. Then $(X,\ast)$ is a semigroup and, for any $x_1,\ldots,x_n\in X$, we have
$$
F(x_1,\ldots,x_n) ~=~ F^*(x_1,\ldots,x_n) ~=~ x_1\,\ast\,\cdots\,\ast\, x_n ~\in ~X,
$$
and hence $(X,F)$ is reducible to the semigroup $(X,\ast)$, a contradiction. It follows that there must exist $a,b\in X$ such that $a\ast b\notin X$, i.e., $a\ast b=e$.
\end{proof}

\begin{proposition}\label{prop:2IfPArtOk76}
Let $(X,F)$ be an $n$-ary semigroup, where $n$ is even. If we can adjoin a neutral element to $(X,F)$, then it is reducible.
\end{proposition}

\begin{proof}
Suppose on the contrary that $(X,F)$ is an $n$-ary IN-semigroup and let $(X^*,F^*)$ and $(X^*,\ast)$ be the semigroups defined in Lemma~\ref{lemma:2ExtABe44}. Then, there exist $a,b\in X$ such that $a\ast b=e$. We then have \begin{eqnarray*}
e &=& \overbrace{e\ast\,\cdots\,\ast e}^{n/2} ~=~ (a\ast b)\ast\,\cdots\,\ast (a\ast b)\\
&=& F^*(\underbrace{a,b,\ldots,a,b}_{n}) ~=~ F(\underbrace{a,b,\ldots,a,b}_{n}) ~\in X\, ,
\end{eqnarray*}
a contradiction (since $e\notin X$).
\end{proof}

\begin{corollary}\label{cor:even}
If $n$ is even, then a neutral element can be adjoined to an $n$-ary semigroup if and only if it is reducible.
\end{corollary}

%------------------------------------------------------------------------------------
\section{A Characterization of all the $n$-ary IN-Semigroups}

In the following theorem, we provide a characterization of the family of all $n$-ary IN-semigroups, i.e., all the $n$-ary semigroups that, while irreducible, admit the adjunction of a neutral element. We first introduce a special class of monoids.

\begin{definition}\label{de:Wmo55}
We say that a semigroup $(M,\ast)$, endowed with a neutral element $e$, is a \emph{W-monoid} if there exists an element $a\in M$ such that the following three conditions are satisfied:
\begin{itemize}
\item[(W1)] for any $x,y\in M$, we have
    $$
    x\ast y=e\quad\iff\quad \text{($x=a$ and $y=a$)$~$ or $~$($x=e$ and $y=e$)};
    $$
\item[(W2)] for any $x,y\in M$, we have
    $$
    x\ast y=a\quad\iff\quad \text{($x=a$ and $y=e$)$~$ or $~$($x=e$ and $y=a$)};
    $$
\item[(W3)] $a$ is noncentral for $\ast$.
\end{itemize}
Note that condition (W1) ensures the uniqueness of $a$, while condition (W3) implies that $a\neq e$.
\end{definition}

\begin{theorem}\label{thm:3main}
If $(X,F)$ is an $n$-ary IN-semigroup, then $n$ is odd. Adjoining a neutral element $e \notin X$ to $(X,F)$, we obtain an $n$-ary semigroup $(X^*,F^*)$ that is the $n$-ary extension of the W-monoid $(X^*,\ast)$ defined by the equation
$$
x\ast y ~=~ F^*(x,e^{n-2},y)\qquad\text{for $x,y\in X^*$}.
$$
Moreover, $(X^*,\ast)$ is the only such monoid with neutral element $e$.

Conversely, let $(M,\ast)$ be any W-monoid with neutral element $e$, and let us consider its $n$-ary extension $F^*\colon M^n\to M$ defined by
$$
F^*(x_1,\ldots,x_n) ~=~ x_1\ast\,\cdots\,\ast x_n\qquad (x_1,\ldots,x_n\in M).
$$
If $n$ is odd, then the $n$-ary semigroup $(X,F^*|_{X^n})$ with $X=M\setminus\{e\}$ is an $n$-ary IN-semigroup.
\end{theorem}

\begin{proof}
Let us prove the first part of the theorem. Let $(X,F)$ be an $n$-ary IN-semigroup, i.e., $(X,F)$ is irreducible and admits the adjunction of a neutral element $e\notin X$. By Proposition~\ref{prop:2IfPArtOk76}, $n$ must be odd, i.e., $n=2k+1$ for some integer $k\geqslant 1$. By Lemma~\ref{lemma:21-firstL4}, the $n$-ary semigroup $(X^*,F^*)$ obtained from $(X,F)$ by adjoining a neutral element $e\notin X$ is reducible to the binary semigroup $(X^*,\ast)$ defined by
$$
x\ast y ~=~ F^*(x,e^{n-2},y),% ~=~ F^*(x,y,e^{n-2}),
$$
and $(X^*,\ast)$ is the unique reduction that has $e$ as its neutral element.

Now, we only need to prove that $(X^*,\ast)$ is a W-monoid. By Lemma~\ref{lemma:2ExtABe44}, there exist $a,b\in X$ such that $a\ast b=e$. The following claim states that $e$ cannot be expressed as a composition of three elements of $X$.

\begin{claim}\label{cl:as7fd6s8a}
For any $x,y,z\in X^*$ such that $x\ast y\ast z=e$, we have $e\in\{x,y,z\}$.
\end{claim}

{
\renewcommand{\qedsymbol}{$\lozenge$}
\begin{proof}[Proof of Claim~\ref{cl:as7fd6s8a}]
Suppose on the contrary that $x\ast y\ast z=e$ for some $x,y,z\in X$. Since $e$ is a neutral element for $\ast$, we have
\begin{eqnarray*}
F(x,y,z,a,b,\ldots,a,b) &=& F^*(x,y,z,a,b,\ldots,a,b)\\
&=& (x\ast y\ast z) \ast \underbrace{(a\ast b) \ast\,\cdots\,\ast (a\ast b)}_{\text{$k-1$ pairs}} ~=~ e.
\end{eqnarray*}
This is a contradiction since $F$ ranges in $X$ and $e\notin X$.
\end{proof}
}

Let us now show that condition (W1) of Definition~\ref{de:Wmo55} holds. By Claim~\ref{cl:as7fd6s8a}, we see that
$$
e ~=~ (a\ast b)\ast (a\ast b) ~=~ a\ast (b\ast a)\ast b
$$
implies that $b\ast a=e$, since $a\neq e\neq b$. From this, it follows that
$$
a\ast(b\ast b)\ast a ~=~ (a\ast b)\ast (b\ast a) ~=~ e ~=~ (b\ast a)\ast (a\ast b) ~=~ b\ast(a\ast a)\ast b.
$$
Applying Claim~\ref{cl:as7fd6s8a} again, we obtain $a\ast a=b\ast b=e$. We then see that $b=a$, since
$$
b ~=~ e\ast b ~=~ (a\ast b)\ast b  ~=~ a\ast (b\ast b) ~=~ a\ast e ~=~ a.
$$
Now, if $e$ has another factorization $e=c\ast d$ with $c,d\in X$, then we can derive similarly that $c=d$. Moreover, we have
$$
e ~=~ (a\ast a)\ast (c\ast c) ~=~ a\ast (a\ast c)\ast c,
$$
and then $a\ast c=e$ by Claim~\ref{cl:as7fd6s8a}, and hence $a=c$ by the preceding argument. This shows that the only nontrivial factorization of $e$ is $e=a\ast a$, thus proving condition (W1).

Let us now show that condition (W2) holds. If $a=x\ast y$ for some $x,y\in X^*$, then $e=a\ast a=(x\ast y)\ast a$, and hence $e\in\{x,y\}$ by Claim~\ref{cl:as7fd6s8a}, which also implies that $\{x,y\}=\{a,e\}$, as $e$ is the neutral element for $\ast$.

To verify that condition (W3) also holds, suppose for contradiction that $a$ is a central element for $\ast$. Then, for any $x\in X$ and any $k\in\{1,\ldots,n\}$, we have
$$
F(a^{k-1},x,a^{n-k}) ~=~ F^*(a^{k-1},x,a^{n-k}) ~=~ a^{k-1}\ast x\ast a^{n-k} ~=~ x\ast \underbrace{a\ast\,\cdots\,\ast a}_{\text{$n-1$ (even)}} ~=~ x,
$$
which shows that $a$ is a neutral element for $F$. This is a contradiction, since $F$ is irreducible.

Let us now prove the second part of the theorem. Let $(M,\ast)$ be a W-monoid with neutral element $e\in M$, and let $a\in M\setminus\{e\}$ be as defined in Definition~\ref{de:Wmo55}. Let also $n$ be an odd integer, and let $(M,F^*)$ be the $n$-ary extension of $(M,\ast)$.

\begin{claim}\label{cl:ae}
For any $x_1,\ldots,x_n\in M$, we have $x_1\ast\cdots\ast x_n\in \{a,e\}$ if and only if $x_1,\ldots,x_n\in \{a,e\}$. Moreover, in this case we have
$$
x_1\ast\,\cdots\,\ast x_n ~=~
\begin{cases}
a, & \text{if $\abs{\{i:x_i=a\}}$ is odd},\\
e, & \text{if $\abs{\{i:x_i=a\}}$ is even}.
\end{cases}
$$
\end{claim}

{
\renewcommand{\qedsymbol}{$\lozenge$}
\begin{proof}[Proof of Claim~\ref{cl:ae}]
By conditions (W1) and (W2), we have $x_1\ast x_2\in \{a,e\}$ if and only if $x_1,x_2\in \{a,e\}$, and then a routine induction argument shows that
$$
x_1\ast\,\cdots\,\ast x_n\in \{a,e\}\quad\iff\quad x_1,\ldots,x_n\in \{a,e\}.
$$
The second statement of the claim follows from the fact that the subsemigroup $(\{a,e\},\ast)$ is isomorphic to $(\Z_2,+)$ under the isomorphism $e \mapsto 0$, $a \mapsto 1$.
\end{proof}
}

Set $X=M\setminus\{e\}$, and let us show that $X$ is closed under $F^\ast$.

\begin{claim}\label{cl:WellDef53}
For any $x_1,\ldots,x_n\in X$, we have $F^\ast(x_1,\ldots,x_n)\in X$.
\end{claim}

{
\renewcommand{\qedsymbol}{$\lozenge$}
\begin{proof}[Proof of Claim~\ref{cl:WellDef53}]
Suppose on the contrary that there exist $x_1,\ldots,x_n\in X$ such that $F^\ast(x_1,\ldots,x_n)=e$.
By Claim~\ref{cl:ae}, we have $x_1,\ldots,x_n\in \{a,e\}$ and $\abs{\{i:x_i=a\}}$ is even.
Since $n$ is odd, this implies that some $x_i$ equals $e$, contrary to our assumption.
\end{proof}
}

Now we can define the restriction $F=F^*|_{X^n}$, and we will show that the $n$-ary semigroup $(X,F)$ is an $n$-ary IN-semigroup. 
It is clear by definition that it admits the adjunction of the neutral element $e$. To see that it is irreducible, we first establish a claim.

\begin{claim}\label{cl:NoNeut}
The $n$-ary semigroup $(X,F)$ has no neutral element.
\end{claim}

{
\renewcommand{\qedsymbol}{$\lozenge$}
\begin{proof}[Proof of Claim~\ref{cl:NoNeut}]
Suppose on the contrary that $(X,F)$ has a neutral element $u\in X$. We then have $F(u^{n-1},a)=a$, that is,
$$
\underbrace{u\ast\,\cdots\,\ast u}_{n-1}\ast\, a ~=~ a,\qquad\text{and hence}\qquad \underbrace{u\ast\,\cdots\,\ast u}_{n-1} ~=~ e.
$$
It follows by Claim~\ref{cl:ae} that $u=a$ (since $u\neq e$), which means that $a$ is a neutral element for $F$. But then, for any $x\in X$, we have (by Claim~\ref{cl:ae}, using that $n-2$ is odd)
$$
x ~=~ F(a,x,a^{n-2}) ~=~ F^*(a,x,a^{n-2}) ~=~ a\ast x\ast\underbrace{a\ast\,\cdots\,\ast a}_{n-2} ~=~ a\ast x\ast a,
$$
which, upon multiplying by $a$ on the left, implies that $a$ is a central element for $\ast$, a contradiction.
\end{proof}
}

Let us now show that the $n$-ary semigroup $(X,F)$ is irreducible. Suppose, on the contrary, that $(X,F)$ is reducible to a semigroup $(X,\circ)$. Then, for any $x\in X$, we have
\begin{eqnarray*}
\overbrace{a\circ\,\cdots\,\circ a}^{n-1} \circ\, x &=& F(a^{n-1},x) ~=~ F^*(a^{n-1},x)\\
&=& \underbrace{a\ast\,\cdots\,\ast a}_{n-1}\ast\, x ~=~ e\ast x ~=~ x,
\end{eqnarray*}
and similarly,
$$
x\circ \, \underbrace{a\circ\,\cdots\,\circ a}_{n-1} ~=~ x\ast\, \underbrace{a\ast\,\cdots\,\ast a}_{n-1} ~=~ x.
$$
Therefore, $\overbrace{a\circ\,\cdots\,\circ a}^{n-1}$ is a neutral element for $(X,\circ)$ and hence also for $(X,F)$. This contradicts Claim~\ref{cl:NoNeut} and therefore completes the proof of the theorem.
\end{proof}

%~ The following example demonstrates that the real ternary semigroup mentioned in the Introduction is not a ternary IN-semigroup, thereby showing that not all irreducible $n$-ary semigroups are $n$-ary IN-semigroups.

%~ \begin{example}\label{ex:ternary2}
%~ It is easy to see that the irreducible ternary semigroup $(\R,F)$ defined by
%~ $$
%~ F(x_1,x_2,x_3) ~=~ x_1-x_2+x_3\qquad\text{for $x_1,x_2,x_3\in\R$},
%~ $$
%~ is not a ternary IN-semigroup, which means that it does not admit the adjunction of a neutral element. Indeed, suppose on the contrary that we can adjoin a neutral element $e\notin\R$ to it. By Theorem~\ref{thm:3main}, the resulting semigroup is then the ternary extension of a unique W-monoid $(X^*,\ast)$, with neutral element $e$ and noncentral element $a$. Now, let $u\in\R\setminus\{a\}$ and define $v=u\ast u$. By condition (W1), it follows immediately that $v\in\R$. For any $x\in\R\setminus\{v\}$, we then have
%~ $$
%~ x\ast v ~=~ x\ast u\ast u ~=~ x ~=~ u\ast u\ast x ~=~ v\ast x,
%~ $$
%~ and hence
%~ $$
%~ x ~=~ v\ast x\ast v ~=~ 2v-x,
%~ $$
%~ which implies $x=v$, a contradiction.
%~ \end{example}

In the next section, we present examples of W-monoids which, in view of Theorem~\ref{thm:3main}, establish the existence of $n$-ary IN-semigroups for odd $n$. This observation enables us to refute the ‘only if’ direction of Dudek and Mukhin’s statement in \cite[Proposition~2]{DudMuk06} for odd $n$. In fact, the flaw in their argument lies only in their proof of this implication. In our notation, they erroneously assert that the existence of elements $a,b\in X$ such that $a\ast b=e$ (see Lemma~\ref{lemma:2ExtABe44}) contradicts the assumption that $e\notin X$.

%------------------------------------------------------------------------------------
\section{Constructions of W-monoids}

In this final section, we outline a procedure for constructing W-monoids via specific ideal extensions of semigroups. We begin by recalling some fundamental definitions. For background, the reader may consult, e.g., Clifford and Preston~\cite{CliPre61}, Petrich~\cite{Pet73}, and Rees~\cite{Ree40}.

Given a semigroup $(S,\ast)$, a subset $I\subset S$ is called an \emph{ideal} of $(S,\ast)$ if, for all $i\in I$ and $x\in S$, we have $i\ast x\in I$ and $x\ast i\in I$. The \emph{Rees congruence} associated with $I$ is the relation $\sim$ on $S$ defined by $x\sim y$ if and only if $x,y\in I$ or $x=y$. The corresponding quotient $S/I$ is referred to as the \emph{Rees quotient}.

An \emph{ideal extension} of a semigroup $S$ by a semigroup $T$ with a zero element is a semigroup $\Sigma$ that contains $S$ as an ideal and for which the Rees quotient $\Sigma /S$ is isomorphic to $T$.

Using these concepts, we can easily derive the following proposition, which describes the monoids satisfying conditions (W1) and (W2) of Definition~\ref{de:Wmo55} in terms of Rees quotients and ideal extensions. The proof is straightforward and is therefore omitted.

\begin{proposition}\label{prop:3MonRees}
Let $(M,\ast)$ be a monoid with neutral element $e$ and let $a\in M\setminus\{e\}$. Then the following assertions are equivalent.
\begin{itemize}
\item[(i)] The monoid $(M,\ast)$ satisfies conditions (W1) and (W2) of Definition~\ref{de:Wmo55}.
\item[(ii)] The subset $I=M\setminus\{a,e\}$ is an ideal of $(M,\ast)$ and the Rees quotient $M/I$ is isomorphic to the set $T=\{-1,0,1\}$ endowed with the usual multiplication, where the equivalence class $[a]$ corresponds to $-1$.
\item[(iii)] $(M,\ast)$ is an ideal extension of a semigroup $S$ by the semigroup $T=\{-1,0,1\}$ endowed with the usual multiplication.
\end{itemize}
\end{proposition}

According to Proposition~\ref{prop:3MonRees}, to establish the existence of W-monoids, it suffices to construct an ideal extension $(M,\ast)$ of a semigroup $S$ by the semigroup $T=\{a,0,e\}\cong\{-1,0,1\}$, in which the element $a$ is noncentral for $\ast$. To this end, the following theorem adapts a result due to Clifford~\cite{Cli50} (see also Clifford and Preston~\cite[Theorem 4.19]{CliPre61}). For simplicity, we denote the operations in both $T$ and $S$ by concatenation.

\begin{theorem}\label{thm:42CliffPres419}
Any semigroup homomorphism $h\,\colon {T^* =\{-1,1\}\cong\Z_2}\to S$ defines an ideal extension $(M,\ast)$ of $S$ by $T=\{-1,0,1\}$, where $M=S\cup T^*$ and the operation $\ast$ is defined by
$$
x\ast y ~=~
\begin{cases}
x{\,}y, & \text{if $x,y\in T^*$},\\
h(x){\,}y, & \text{if $x\in T^*$ and $y\in S$},\\
x{\,}h(y), & \text{if $x\in S$ and $y\in T^*$},\\
x{\,}y, & \text{if $x,y\in S$}.
\end{cases}
$$
Moreover, if $S$ has a neutral element, then every ideal extension of $S$ by $T$ arises in this manner.
\end{theorem}

As stated, Theorem~\ref{thm:42CliffPres419} provides a straightforward method for constructing a wide range of W-monoids. The following example illustrates a particularly simple instance of such monoids.

\begin{example}\label{ex:Order2}
Let $(S,\circ)$ be a monoid with a neutral element $\id$, and let $A\in S$ be an element satisfying $A\circ A=\id$. It is straightforward to verify that the map $h\,\colon T^* =\{-1,1\}\to S$ defined by $h(1)=\id$ and $h(-1)=A$ is a semigroup homomorphism. By Theorem~\ref{thm:42CliffPres419}, this induces an ideal extension $(M,\ast)$ of $S$ by $T$, where $M=S\cup T^*$. Moreover, if $A$ is noncentral in $S$, then it remains noncentral in $M$, and consequently, $(M,\ast)$ is a W-monoid. As a concrete example, one can consider the general linear group $S=\mathrm{GL}_2(\R)$ of all invertible $2\times 2$ matrices over $\R$ and the involutive matrix $A=\mathrm{diag}(1,-1)$.
\end{example}

\begin{remark}\label{rem:Order2}
If the semigroup $(M,\ast)$ constructed by Theorem~\ref{thm:42CliffPres419} is a monoid, then it is not difficult to see that $h(1)$ must be a neutral element for $S$.
Conversely, according to the last statement of the theorem, all W-monoids $(M,\ast)$ such that $S=M\setminus\{a,e\}$ has a neutral element are necessarily constructed along the lines of Example~\ref{ex:Order2}.
\end{remark}

We now observe that a result by Yoshida~\cite{Yos65} (see also Petrich~\cite[Theorem~III.2.2]{Pet73}) offers a description of all ideal extensions, based on the concept of the translational hull. To proceed, we first recall some fundamental concepts related to semigroups (see \cite[p.~63]{Pet73}). Given a semigroup $(S,\circ)$, 
\begin{itemize}
\item a \emph{left translation} of $S$ is a map $\lambda\colon S\to S$ satisfying
$$
\lambda(x\circ y) ~=~ \lambda(x)\circ y\qquad (x,y\in S);
$$
\item a \emph{right translation} of $S$ is a map $\rho\colon S\to S$ satisfying
$$
\rho(x\circ y) ~=~ x\circ \rho(y)\qquad (x,y\in S);
$$
\item a \emph{bitranslation} of $S$ is a pair $(\lambda,\rho)$ such that
$$
x\circ \lambda(y) ~=~ \rho(x)\circ y\qquad (x,y\in S),
$$
where $\lambda$ is a left translation and $\rho$ is a right translation of $S$;
\item the \emph{translational hull} of $S$ is the semigroup $\Omega(S)$ of all the bitranslations.
\end{itemize}
Associativity implies that the pair $(\lambda,\rho)$ with $\lambda(y) = a \circ y$ and $\rho(x) = x \circ a$ is a bitranslation for any given $a \in S$. 

Applying Yoshida’s result, we obtain the following characterization of the class of W-monoids.

\begin{theorem}\label{thm:44-BasPetY}
Let $(S,\circ)$ be a semigroup and consider a bitranslation $(L,R)$ of $S$ satisfying $L^2=R^2=\id$, $LR=RL$, and $L\neq R$. Then, the structure $(M,\ast)$, where $M=S\cup\{a,e\}$ for some $a,e\notin S$, with $e$ neutral for $\ast$, and
$$
x\ast y ~=~
\begin{cases}
x\circ y, & \text{if $x,y\in S$},\\
L(y), & \text{if $x=a$ and $y\in S$},\\
R(x), & \text{if $x\in S$ and $y=a$},\\
e, & \text{if $x=a$ and $y=a$}.
\end{cases}
$$
is a W-monoid. Moreover, all W-monoids can be constructed in this manner.
\end{theorem}

\begin{proof}
By Proposition~\ref{prop:3MonRees}, a W-monoid is an ideal extension $(M,\ast)$ of a semigroup $(S,\circ)$ by $T=\{a,0,e\}$, with the additional properties that $e$ is a neutral element for $\ast$ and $a$ is noncentral for $\ast$. We first observe that $T^* =\{a,e\}$ forms a group, and hence the \emph{ramification set} \cite[p.~68]{Pet73} of $T$ is empty. Consequently, such an extension is determined by a homomorphism $\theta\colon T^*\to\Omega(S)$, which maps $T^*$ onto a set of permutable bitranslations. Since $e$ is the neutral element for $\ast$, we must have $\theta(e)=(\id,\id)$. Moreover, setting $\theta(a)=(L,R)$, the map $\theta$ is an homomorphism if and only if $L^2=R^2=\id$. In addition, it has \emph{permutable values} \cite[p.~68]{Pet73} if and only if $LR=RL$. Note also that the condition $L\neq R$ is equivalent to $a$ being noncentral for $\ast$. Finally, Yoshida's result (see \cite[Theorem~III.2.2]{Pet73}) guarantees that $(M,\ast)$ is a W-monoid and that all W-monoids can be constructed in this way.
\end{proof}

We can actually prove Theorem~\ref{thm:44-BasPetY} without making use of Yoshida's theory or any semigroup theory beyond the basic definition of a bitranslation. We now present such an elementary proof.

\begin{proof}[Self-contained proof of Theorem~\ref{thm:44-BasPetY}]
Assume first that $(M,\ast)$ is a W-monoid with neutral element $e$, and let $a \in M$ be the element provided by Definition~\ref{de:Wmo55}. Conditions (W1) and (W2) guarantee that $(S,\circ)$, with $S := M \setminus \{a,e\}$, is a subsemigroup of $M$ (here $\circ$ denotes the restriction of $\ast$ to $S$ in order to be consistent with the notation of the theorem).
Let $L(y) = a \ast y$ and $R(x) = x \ast a$; then $(L,R)$ is a bitranslation of $M$. 
It is clear from (W1) and (W2) that $S$ is closed under $L$ and $R$, hence their restriction to $S$ constitutes a bitranslation of $(S,\circ)$.
Associativity immediately implies $LR=RL$, while associativity together with $a \ast a = e$ shows that $L^2 = R^2 = \id$.
Condition (W3) ensures that $L \neq R$, and the formula for $\ast$ stated in the theorem follows from the definition of $\circ$, $L$, and $R$.

Conversely, if $(M,\ast)$ is constructed as described in the theorem, then a simple verification yields that $\ast$ is associative (using the associativity of $\circ$, the definition of a bitranslation, and the assumptions $L^2=R^2=\id$ and $LR=RL$).
The definition of the operation $\ast$ shows that $x \ast y \in \{a,e\}$ can happen only if $x,y \in \{a,e\}$, and then (W1) and (W2) follow from $a \ast a = e$ and from the fact that $e$ is the neutral element for $(M,\ast)$.
Finally, $L \neq R$ implies condition (W3).
\end{proof}

In the following example, we construct a W-monoid using Theorem~\ref{thm:44-BasPetY} that cannot be obtained via Theorem~\ref{thm:42CliffPres419}.

\begin{example}
Let $(X,\diamond)$ be a monoid with a neutral element $\id$ and let $i,j\in X\setminus\{\id\}$ be elements satisfying $i\diamond i=j\diamond j=\id$. Consider also the semigroup $(S,\circ)$, where $S=X\times X$ and $\circ\,\colon S^2\to S$ is the binary operation  defined by
$$
(x,y)\circ (x',y') ~=~ (x\diamond x',y').
$$
It is then easy to verify that the maps $L\colon S\to S$ and $R\colon S\to S$ defined by
$$
L((x,y)) ~=~ (i\diamond x,y)\quad\text{and}\quad R((x,y)) ~=~ (x\diamond i,y\diamond j)
$$
satisfy the properties stated in Theorem~\ref{thm:44-BasPetY}. For instance, for any $x,y\in X$, we have
$$
R^2((x,y)) ~=~ R(x\diamond i,y\diamond j) ~=~ (x\diamond i\diamond i,y\diamond j\diamond j) ~=~ (x,y).
$$
Using Theorem~\ref{thm:44-BasPetY}, we immediately obtain a W-monoid. If $\abs{X} \geqslant 2$, then $S$ does not have a neutral element, hence the corresponding W-monoid cannot be obtained via Theorem~\ref{thm:42CliffPres419} (see Remark~\ref{rem:Order2}).
As a concrete example, we can take $X=\{0,1\}\cong\Z_2$ and $i=j=1$, and the binary operation $\diamond$ is the addition modulo $2$.
\end{example}

%------------------------------------------------------------------------------------

\end{document}